\RequirePackage{fix-cm}
\documentclass[smallextended]{svjour3}       
\smartqed  
\usepackage{graphicx}
\usepackage{amsmath}
\usepackage{tikz}
\begin{document}

\title{Perfect Elimination Orderings for Symmetric  Matrices
\thanks{This work is dedicated to the memory of Michel Deza, with gratitude for his support during the early career of the first author. }
}


\author{Monique Laurent         \and
        Shin-ichi Tanigawa 
}


\institute{M. Laurent \at
              Centrum Wiskunde \& Informatica (CWI), Science Park 123, 1098 XG Amsterdam, The Netherlands. \\ Tilburg University, P.O. Box 90153, 5000 LE Tilburg, The Netherlands. \\
              \email{M.Laurent@cwi.nl}           
           \and
           S. Tanigawa \at
              Centrum Wiskunde \& Informatica (CWI), Science Park 123, 1098 XG Amsterdam, The Netherlands. \\ Research Institute for Mathematical Sciences, Kyoto University, Kitashirakawa-Oiwaketyo, Sakyo-ku, Kyoto, 606-8502, Japan\\
              \email{shinichi.tanigawa@gmail.com}
}

\date{Received: date / Accepted: date}

\maketitle

\begin{abstract}
We introduce a new class of structured symmetric matrices by extending the notion of perfect elimination ordering from graphs to weighted graphs or matrices. This offers a common framework capturing common vertex elimination orderings of monotone families of chordal graphs, Robinsonian matrices and ultrametrics. We give  a structural characterization for matrices that admit perfect elimination orderings in terms of forbidden substructures generalizing chordless cycles in graphs.
\keywords{Chordal graph \and Perfect elimination ordering \and Unit interval graph \and Ultrametric \and Shortest path metric \and Robinson matrix}
\end{abstract}

\section{Introduction} 
We introduce a new class of structured matrices by ways of {\em perfect elimination orderings}, an extension to weighted graphs of the classical notion of perfect elimination ordering for graphs. This  offers a common framework for the study of (adjacency matrices of) chordal graphs (and  their metric powers)  
as well as for  ultrametrics 
 and Robinsonian (dis)similarity matrices.

\smallskip 
Recall that a graph $G=(V,E)$ is {\em chordal} when it does not contain a chordless cycle of length at least 4, where a  cycle $C=(v_1,\cdots, v_p)$ in $G$ is said to be {\em chordless} if $C$ is an induced subgraph of $G$, i.e., if none of the pairs $\{v_i,v_j\}$ for $|i-j| \ge 2$ (indices taken modulo $p$) is an edge of $G$.
Chordal graphs  appear as tractable  or well-behaved cases in many  optimization problems 
 (see, e.g., \cite{Gavril,Grone}). This is often due to their equivalent characterization in terms of perfect elimination orderings.
A linear order $\pi$ of $V$ is called a {\em perfect elimination ordering} of $G$ if, for any vertices $x,y,z\in V$ such that 
$x<_\pi y<_\pi z$, $\{x,y\},\{x,z\}\in E$ implies $\{y,z\}\in E$. 
It is a well known fact that $G$ is chordal if and only if $G$ has a perfect elimination ordering~\cite{Dirac61,FG}.

\smallskip 
In this paper we extend this  notion  of vertex ordering to weighted graphs, aka symmetric matrices.
Throughout $\mathcal S^V$ is the set of symmetric matrices indexed by the set $V=[n]$.
Given a matrix $A=(A_{xy})\in \mathcal S^V$ we say that a linear order $\pi$ of $V$ is a {\em perfect elimination ordering of $A$} if it satisfies   the following {\em three-points condition}
\begin{equation}\label{eqineq}
A_{yz}\ge \min\{A_{xy},A_{xz}\}  \ \text{ for all } x,y,z\in V \text{ with } x<_{\pi} y<_{\pi} z.
\end{equation}
Note that the diagonal entries do not play a role in this definition.
When $A=A_G$ is the adjacency matrix of a graph $G$  both notions of perfect elimination orderings of $A_G$ and perfect elimination orderings of $G$ coincide. 


\smallskip 
Given a distance space $(V,d)$ and its associated distance matrix $D\in \mathcal S^V$ (with entries $D_{xy}=D_{yx}=d_{xy}$ for $x\ne y$ and $D_{xx}=0$ for $x,y\in V$), a perfect elimination ordering of the matrix $A=-D$ is a linear order $\pi$ of $V$ satisfying
\begin{equation}\label{eqD}
D_{yz} \le \max\{D_{xy},D_{xz}\}  \ \text{ for all } x,y,z\in V \text{ with }  x<_{\pi} y<_{\pi} z.
\end{equation}
Recall that $(V,d)$ is an ultrametric if the inequality in (\ref{eqD}) holds for {\em all} elements $x,y,z\in V$. In other words we have the following connection.

\begin{lemma}\label{lemultrametric}
Let $(V,d)$ be a  distance space  with distance matrix $D$. Then $(V,d)$  is an ultrametric if and only if every linear ordering of $V$ is a perfect elimination ordering of the matrix $-D$.
\end{lemma}

Another special class of matrices admitting a perfect elimination ordering arises from Robinsonian matrices.
A symmetric matrix $A$  is called a {\em Robinsonian similarity matrix} if 
there exists a linear order $\pi$ of $V$ satisfying the following three-points condition:
\begin{equation}
\label{eq:rob}
A_{xz}\le \min\{A_{xy},A_{yz}\}  \ \text{ for all } x,y,z\in V \text{ with } x<_\pi y<_\pi z;
\end{equation}
such an ordering is then called a {\em Robinson ordering} of $A$.
In the context of distances,  a  matrix $D$ is called a {\em Robinsonian dissimilarity matrix}
 when $A=-D$ is a Robinsonian similarity matrix. 
Robinsonian matrices have a long history and play an important role in classification problems in data science, in particular in ranking problems \cite{FAV} and in  the seriation problem (introduced by the archeologist Robinson for chronological dating) (see, e.g., \cite{Liiv}). There the goal is to order (seriate)  a set of objects, given through their pairwise (dis)similarities, in such a way that  two objects are ranked close to each other if they have a large correlation/similarity (or a small dissimilarity).

It is a classical observation by Roberts~\cite{Roberts69} that the adjacency matrix of a graph $G$ is Robinsonian if and only if $G$ is a unit interval graph, i.e.,   its vertices can be labeled by unit intervals on a line so that adjacent vertices receive intersecting intervals.
Clearly, the condition (\ref{eq:rob}) implies (\ref{eqineq}) and thus  any Robinson ordering of $A$  is a perfect elimination ordering of $A$.
For adjacency matrices of graphs, this corresponds to the  fact that unit interval graphs are chordal graphs. 
So Robinsonian matrices are weighted graph analogues of unit interval graphs and this fact formed the motivation for the present work to investigate weighted analogues of  chordal graphs.

\smallskip
There is a well known structural characterization of unit interval graphs  in terms of minimal forbidden substructures (namely,
claws and asteroidal triples; see \cite{Gardi,Roberts69}). 
An analogous  structural characterization was given in \cite{Laurent17} for  Robinsonian matrices (by extending the notion of asteroidal triple to weighted graphs). 
For chordal graphs the minimal forbidden substructures are the chordless cycles. This raises the natural question of understanding the minimal forbidden substructures for the class of matrices   admitting   a perfect elimination ordering.
A main contribution of this note is to provide such a structural characterization, in terms of weighted chordless walks, a new key notion which we will introduce below (see Theorem \ref{thmA}).

\medskip
The paper is organized as follows.  Sections \ref{sec21} and \ref{sec22}  contain definitions and  preliminary  results about vertex elimination orderings and   simplicial vertices. In Section \ref{sec23} we present our main structural result (Theorem \ref{thmA}) for matrices admitting a perfect elimination ordering. In Section \ref{sec24} we discuss related notions: common perfect elimination orderings of powers of chordal graphs, distance-preserving elimination orderings of shortest path distance matrices, and conclude with a brief discussion of other graph properties that could be extended to matrices and of related recognition algorithms.
The last Section \ref{secproof} is devoted to the proof of our main structural result in Theorem \ref{thmA}.

\section{Perfect Elimination Orderings of Matrices}\label{sec2}
\subsection{Perfect elimination orderings and simplicial elements}\label{sec21}

Given a graph $G=(V,E)$, recall that  $v\in V$ is  a {\em simplicial vertex} of $G$ if  its  neighbors  form a clique of $G$. Then an order $\pi$ of $V$ is a perfect elimination ordering of $G$  if and only if each vertex $v$ is simplicial in $G[\{x\in V: v\le_\pi x\}]$, the subgraph of $G$ induced by the nodes coming after $v$ in $\pi$.
In the same way, given a   matrix $A\in \mathcal S^V$, an element $v\in V$ is said to be {\em simplicial in $A$} 
if 
\begin{equation}\label{eqsimplicial}
A_{yz}\ge \min\{A_{vy},A_{vz}\}\ \text{ for all distinct } y,z\in V\setminus \{v\}.
\end{equation}
Then an order $\pi$ of $V$ is a perfect elimination ordering of $A$ precisely when each  $v\in V$ is simplicial in $A[\{x\in V: v\le_\pi x\}]$,   the principal submatrix of $A$ indexed by the elements coming after $v$ in $\pi$. 
We next observe that  simplicial elements are precisely those coming first in some perfect elimination ordering.


\begin{lemma}
Assume $A$ has a perfect elimination ordering and let $v\in V$. Then, $v$ is simplicial for $A$ if and only if there exists a perfect elimination ordering of $A$ with $v$ as first element.
\end{lemma}

\begin{proof}
The `if' part is clear. We show the `only if part' by induction on the size $n$ of $A$.
The case $n=3$ is clear. 
Assume now that $a$ is simplicial for $A$ and consider a perfect elimination ordering $\pi$ of $A$ starting at $b\neq a$.  Then we know that  $b$ is simplicial in $A$. Consider the submatrix $A'$ induced by $V\setminus \{b\}$. Then $A'$ still has a perfect elimination ordering and $a$ is still simplicial in $A'$. Hence, by the induction assumption, there exists a perfect elimination ordering $\pi'$ of $A'$ starting at $a$, say $\pi'=(a,u,\cdots, w).$
We consider the ordering $\tilde \pi=(a,b,u,\cdots, w)$ obtained by inserting $b$ between $a$ and $u$. We claim that $\tilde \pi$ is a perfect elimination ordering of $A$, that is,
$A_{yz}\ge \min\{A_{xy},A_{xz}\}$ for all $x<_{\tilde \pi} y<_{\tilde \pi} z$. This is true if $b\not\in \{x,y,z\}$ since $\pi'$ is a perfect elimination ordering.
Assume now that $b\in \{x,y,z\}$. Then $b\ne z$ (as $b$ is second in $\tilde \pi$). 
If $b=x$ then the desired inequality follows from the fact that $b$ is simplicial in $A$.
Finally, if $b=y$ then $x=a$ and the desired inequality follows from the fact that $a$ is simplicial in $A$. \qed
\end{proof}

Note that finding a simplicial element in $A$ can be done in $O(n^3)$ operations and thus one can find a perfect elimination ordering of $A$ in $O(n^4)$ operations (or decide that none exists), where $n=|V|$.





\subsection{Common perfect elimination orderings }\label{sec22}

Let $\alpha_0< \alpha_1<\alpha_2 <\cdots < \alpha _L$ denote the distinct values taken by the entries of a matrix $A\in \mathcal S^V$
and, for $\ell =0,1,\cdots, L$, define its  {\em level graph} $G_\ell=(V,E_\ell)$, whose edges are the pairs $\{x,y\}$ with $A_{xy}\ge \alpha_\ell$. Thus,  $(V,E_0)$ is the complete graph on $V$ (i.e., $A_{G_0}$ is the all-ones matrix)  and $E_L\subseteq \cdots \subseteq E_1$. 
It is easy to check that (up to shifting  all entries of $A$ by $\alpha_0$ and assuming all its diagonal entries  are zero) 
 $A$ can be decomposed as a conic combination of the adjacency matrices of its level graphs:
\begin{equation}\label{eqdec}
A- \alpha_0 A_{G_0}= 
 \sum_{\ell=1}^L (\alpha_{\ell}-\alpha_{\ell-1}) A_{G_{\ell}}
\end{equation}
 As a direct  application we have the following characterization.

\begin{lemma}
A matrix $A$ has a perfect elimination ordering if and only if 
 there exists  an ordering $\pi$ of $V$ which is a common perfect elimination ordering of all the level graphs of $A$.
\end{lemma}

In other words, a necessary condition for $A$ to have a perfect elimination ordering is that all its level graphs be chordal, however for finding  a perfect elimination ordering of $A$ one needs to find  an ordering which is a perfect elimination ordering {\em simultaneoulsy}  for all its level graphs.


Clearly we may assume without loss of generality that $\alpha_0=0$.  Moreover,  the  exact values of $\alpha_0,\cdots,\alpha_L$  are not important (as long as they are strictly increasing).
For instance, $A$ has a perfect elimination ordering if and only if the matrix 
$\widetilde A= \sum_{\ell=1}^L A_{G_{\ell}}$ does too. Hence the question  whether a matrix has a perfect elimination ordering is equivalent to asking whether a  finite {\em monotone} family of graphs $G_1\supseteq G_2\supseteq \dots \supseteq G_L$  admits a common perfect elimination ordering. We will come back to this in Section~\ref{secchordalpower}.

Finally observe  that any {\em  arbitrary} order of $V$ is a perfect elimination ordering of $A$ precisely when all its level graphs are disjoint unions of cliques. Such (similarity) matrices $A$ correspond thus to distance matrices $D$ of ultrametrics, via the correspondence $A=-D$.

\subsection{Structural characterization of matrices with  perfect elimination orderings}\label{sec23}

We now describe the structural obstructions for the symmetric matrices admitting a perfect elimination ordering. First we introduce some notation.

A {\em walk} is an ordered sequence $W=(v_0,v_1,\cdots, v_p)$ of elements of $V$. 
Then we set $V(W)=\{v_0,v_1,,\cdots,v_p\}$, $I(W)=\{v_i: 1\le i\le p-1\}$ is the set of  {\em internal} elements of $W$,  $v_0$ and $v_p$ are its {\em end points}.
 The walk $W$ is said to be {\em closed} if $v_0=v_p$ and the walk $W$ is said to be {\em self-contained}  if $I(W)=V(W)$.
A closed walk is called a {\em cycle}
if $v_0,v_2,\cdots, v_{p-1}$ are all distinct.
%



The following  notion 
will play a key role in our structural characterization:
A walk $W=(v_0,v_1,\cdots, v_p)$ is said to be {\em weighted chordless} in $A$ if 
\begin{equation}\label{eqwalk}
A_{v_{i-1}v_{i+1}}<\min\{A_{v_{i-1}v_{i}},A_{v_{i+1}v_{i}}\} \ \text{ for all } 1\le i\le p-1.
\end{equation}
In addition,  $W$ is said to be a  {\em weighted chordless cycle} in $A$ if $W$ is a cycle  which, in addition to (\ref{eqwalk}), also satisfies the inequality
\begin{equation}\label{eqcycle}
A_{v_{p-1}v_1} <\min\{A_{v_{p-1}v_0},A_{v_0v_1}\}.
\end{equation}
Hence a walk $(y,x,z)$ is  weighted chordal   precisely when the  triple $(x,y,z)$  violates the inequality in (\ref{eqineq}) and thus its internal element $x$ cannot come before both $y,z$ in any perfect elimination ordering of $A$.

 It is useful to compare with the  notion of chordless cycle in graphs. 
Let $A=A_G$ be the adjacency matrix of a graph $G=(V,E)$. Then, a walk $W=(v_1,\cdots,v_p)$ is   {\em weighted chordless}  in $A_G$ precisely when all the 2-chords $\{v_i,v_{i+2}\}$ ($1\le i\le p-2$) are not edges of $G$. Therefore, either $W$ is an induced walk in $G$ (i.e., none of the chords $\{v_i,v_j\}$ ($1\le i$, $i+2\le j\le p$) is an edge of $G$), or $W$ contains a chordless cycle of $G$ (meaning $V(W)$ contains a subset inducing a chordless cycle in $G$). In particular, if $W$ is  a weighted chordless cycle in $A_G$ then $W$ is equal to or contains a chordless cycle of $G$.

By definition, chordal graphs are exactly the graphs that have no chordless cycle of length at least 4.
It is natural to ask whether a similar structural characterization holds for matrices.
We start with some easy observations.

\begin{lemma}
\label{prop:cycle}
Consider a matrix $A\in \mathcal S^V$. If  (i)  $A$ has a perfect elimination ordering then  (ii) $A$ has no weighted chordless cycle, which in turn implies that (iii) every level graph of $A$ is a chordal graph.
\end{lemma}

\begin{proof} (i) $\Longrightarrow$ (ii): 
Assume $\pi$ is a perfect elimination ordering of $A$,  $W$ is a weighted chordless cycle in $A$ and $v_i$ is the element of $W$ coming first  in $\pi$. As $(v_{i-1},v_i,v_{i+1})$ is a weighted chordless walk we get a contradiction. \\ 
(ii) $\Longrightarrow$ (iii): Assume  $C=(v_1,\cdots,v_p)$ is a chordless cycle in some level graph $G^{(l)}$ of $A$, i.e., $A_{v_iv_{i+1}}\ge \alpha_l$ for $i=1,\cdots,p$, while $A_{v_iv_j}<\alpha_l$ whenever $|i-j|\ge 2$. Then $C$ is a weighted chordless cycle in $A$, contradicting (ii). \qed
\end{proof}
The reverse implications are not true in general. See Figure~\ref{fig:1} for examples.
%
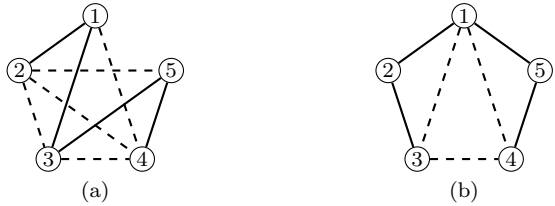
\begin{figure}[ht]
\centering
\begin{minipage}{0.4\textwidth}
\centering
\begin{tikzpicture}[scale=0.7]

\def \n {5}
\def \radius {1.5cm}
\def \margin {10} 

\foreach \s in {1,...,\n}
{
  \node[draw, circle, inner sep=0.9pt,minimum size=1pt] (\s) at ({360/\n * (\s - 1)+90}:\radius) {$\s$};
}
\foreach  \x/\y in {1/2,1/3,3/5,4/5}
{
\draw[thick] (\x) -- (\y);
}
\foreach \x/\y in {1/4,2/4,2/3,3/4,2/5}
{
\draw[thick,dashed] (\x) -- (\y);
}
\end{tikzpicture}
\par
(a)
\end{minipage}
\begin{minipage}{0.4\textwidth}
\centering
\begin{tikzpicture}[scale=0.7]

\def \n {5}
\def \radius {1.5cm}
\def \margin {10} 

\foreach \s in {1,...,\n}
{
  \node[draw, circle, inner sep=0.9pt,minimum size=1.5pt] (\s) at ({360/\n * (\s - 1)+90}:\radius) {$\s$};
}
\foreach  \x/\y in {1/2,2/3,4/5,1/5}
{
\draw[thick] (\x) -- (\y);
}
\foreach \x/\y in {1/3,1/4,3/4}
{
\draw[thick,dashed] (\x) -- (\y);
}
\end{tikzpicture}
\par
(b)
\end{minipage}
\caption{Visualization of  two $\{0,1,2\}$-symmetric matrices of size 5 as weighted graphs, 
where  bold edges have weight 2,  dashed edges have weight 1 and non-edges have weight 0.
(a) An example of matrix which   has  no weighted chordless cycle 
 (because $(2,1,3), (1,2,5), (1,3,5), (1,4,5), (3,5,4)$ are all the chordless 2-walks  and  they cannot be concatenated to build a weighted chordless  cycle) and also   no  simplicial vertex (and thus no perfect elimination ordering). (b) An example where all  level graphs are chordal, but $(1,2,3,4,5)$  is  a weighted chordless cycle.
}
\label{fig:1}
\end{figure}

Although the matrix in Figure~\ref{fig:1}(a) has no weighted chordless cycle, it yet contains a forbidden structure for perfect elimination orderings, namely  $(1,4,5,3,1,2,5)$ is a self-contained weighted chordless walk in the matrix.
Recall that a  walk  $W$ is {\em self-contained} if  $V(W)=I(W)$. 
More generally,   a family $\{W_1,\cdots,W_k\}$ of  walks  is {\em self-contained} if $\cup_{h=1}^k V(W_h)=\cup_{h=1}^k I(W_h)$.
Figure~\ref{fig:2} gives an example having no simplicial vertex and also no self-contained weighted chordless walk,
so forbidding a single self-contained chordless walk is not sufficient to guarantee a perfect elimination ordering;m we need to forbid families of them. 

\begin{figure}[ht]
\centering
\begin{tikzpicture}[scale=0.7]

\def \n {6}
\def \radius {1.5cm}
\foreach \s in {1,...,\n}
{
  \node[draw, circle, inner sep=0.9pt,minimum size=1.5pt] (\s) at ({360/\n * (\s - 1)+90}:\radius) {$\s$};
}
\foreach  \x/\y in {1/2,1/3,4/6,5/6}
{
\draw[thick] (\x) -- (\y);
}
\foreach \x/\y in {1/4,1/5,2/6,3/6,2/3,3/4,4/5,2/4,2/5,3/5,4/5}
{
\draw[thick,dashed] (\x) -- (\y);
}
\end{tikzpicture}
\caption{Visualization of a $\{0,1,2\}$-symmetric matrix of size 6 as a weighted graph, 
where  bold edges have weight 2, dashed edges have weight 1 and non-edges have weight 0.
This matrix   has no simplicial vertex and no self-contained  weighted chordless walk 
(because $(2,1,3), (1,2,6), (1,3,6), (1,4,6), (1,5,6),  (4,6,5)$ are all the chordless 2-walks  but they cannot be concatenated to a build a self-contained weighted chordless walk).}
\label{fig:2}
\end{figure}
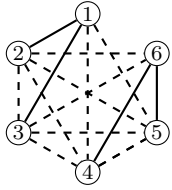

\begin{lemma}\label{prop:1}
If  $A$ has a self-contained family of weighted chordless walks, then $A$ does not have a perfect elimination ordering.
\end{lemma}
\begin{proof}
Same proof as for the implication (i) $\Longrightarrow$ (ii)  in  Lemma~\ref{prop:cycle}.
\qed
\end{proof}


\begin{lemma}\label{lem:simplicial}
Assume $A$ does not contain any self-contained family of weighted chordless walks. Then $A$ has a simplicial vertex.
\end{lemma}

\begin{proof}
Suppose for contradiction that $A$ does not have a simplicial vertex. That is, for any $x\in V$ there exist $y\ne z\in V\setminus \{x\}$ such that 
$A_{yz}<\min\{A_{xy},A_{xz}\},$ 
i.e.,  $P_x:=(y,x,z)$ is a weighted chordless walk.
Then we have a self-contained  family  $\{P_x: x\in V\}$ of weighted chordless walks in $A$. \qed
\end{proof}

\begin{corollary}\label{theochordless}
A  matrix $A$  has a perfect elimination ordering if and only if there does not exist a
self-contained  family of weighted chordless walks in $A$.
\end{corollary}

\begin{proof}
The `only if' part is shown in Lemma~\ref{prop:1}.
We now show the `if part',
using induction on the size $n$  of $A$. 
The base case $n\le 3$ is trivial. So let us assume that $A$ has size $n\ge 4$. In view of Lemma \ref{lem:simplicial}, $A$ has a simplicial vertex $v_1$. 
Consider now the principal submatrix $A_1$ of $A$ indexed by $V\setminus \{v_1\}$. 
By the induction assumption, $A_1$ has a perfect elimination ordering $\pi_1$. Then, appending $v_1$ as first element before $\pi_1$, we get the ordering $\pi=(v_1,\pi_1)$ of $V$, which is a perfect elimination ordering of $A$. This concludes the proof. \qed
%
%
\end{proof}


The argument used  in the proof of 
Lemma~\ref{lem:simplicial} is the matrix analogue of  the well known  fact that a graph has no simplicial vertex if and only if each vertex is  the midpoint of an induced $P_3$ (a path with three vertices).
Dirac's theorem informally says  that some of these induced $P_3$'s can be  assembled to form a chordless cycle. 
As a matrix analogue, we will show that it suffices to exclude  self-contained {\em pairs of two} weighted chordless walks.
This is our main structural characterization result,  which we will prove in Section \ref{secproof} below, since the technical details are more involved.


\begin{theorem}\label{thmA}
A symmetric matrix $A$ has a perfect elimination ordering if and only  if $A$ has no self-contained pair of weighted chordless walks.
\end{theorem}

Note that the matrix from  Figure~\ref{fig:2}  has a self-contained pair of two weighted chordless walks, namely $W_1=(6,2,1,3,6)$ and $W_2=(1,4,6,5,1)$ (note that  $6\in I(W_2)$ and $1\in I(W_1)$).

\subsection{Applications and related concepts}\label{sec24}



\subsubsection{Common perfect elimination orderings of  powers of chordal graphs}\label{secchordalpower}

Given a graph   $G=(V,E)$ let $d_G$ be its shortest path metric, with associated distance matrix $D_G$.
For a positive integer $k$, the {\em $k$-th power} $G^k$ is the graph on $V$, whose edges are the pairs $\{u,v\}$ with distance $d_G(u,v)\le k$. 
So we have a monotone graph family: $G^1\subseteq \cdots \subseteq G^k$.
Duchet~\cite{Duchet84} shows that if $G^k$ is chordal then so is $G^{k+2}$.
Hence if $G$ and $G^2$ are chordal then all  powers of $G$ are chordal.
Dragan et al.~\cite{Dragan92} 
prove that if $G$ and $G^2$ are chordal then they admit a common common perfect elimination ordering (see also \cite[Thm 5]{Brandstadt98}), 
and Brabdst{\"a}d et al.~\cite{Brandstadt96} prove that, for any integers $i_1,\dots, i_k$, $G^{i_1}, \dots, G^{i_k}$ admit a common elimination ordering if 
they are all chordal.
Consequently, the following holds.

\begin{theorem}[Brabdst{\"a}d, Chepoi and Dragan~\cite{Brandstadt96}]\label{thm:power}
If $G$ and $G^2$ are chordal, then all  the powers of $G$ admit a common perfect elimination ordering.
\end{theorem}
This theorem has an interesting implication in our context: the reverse direction of Lemma~\ref{prop:cycle} is true for the shortest path distance matrix $D_G$ of a graph.

\begin{corollary} \label{cor:power}
Let $G$ be an undirected graph and let $D_G$ be its shortest path  distance matrix.
Then the following assertions are equivalent.
\begin{itemize}
\item[(i)]  $-D_G$ has a perfect elimination ordering;
\item[(ii)] $-D_G$ has no weighted chordless cycle;
\item[(iii)] Every level graph of $-D_G$ is chordal;
\item[(iv)] $G$ and $G^2$ are chordal.
\end{itemize}
\end{corollary}
\begin{proof}
This follows from Lemma~\ref{prop:cycle} and Theorem~\ref{thm:power} after observing the correspondence between  the adjacency matrices of the powers of $G$ and the level graphs of $-D_G$.
\qed
\end{proof}
This result does not extend to shortest path distance matrices of weighted graphs.
For this consider the matrix $A$ from Figure~\ref{fig:2} and   define the matrix $D=(3-A_{xy})_{x,y\in [6]}$.
Then $D$ is a shortest path distance matrix (for the weights $D_{xy}$), but  $-D$ has no perfect elimination ordering and   no weighted chordless cycle.

\subsubsection{Distance-preserving elimination orderings}
Here we point out a link between perfect elimination orderings and the following notion of distance-preserving ordering considered by Chepoi \cite{Chepoi98}.
For a graph $G=(V,E)$, a linear ordering $v_1,\dots, v_n$ of $V$ is called a {\em distance-preserving elimination ordering} if for each $i\in [n]$ 
the subgraph $G_i$ of $G$ induced by $\{v_i, \dots, v_n\}$ is {\em isometric}, 
i.e., $d_{G_i}$ coincides with the restriction of $d_G$ to $\{v_i, \dots, v_n\}$.
This notion can be naturally generalized to weighted graphs: 
Given  nonnegative edge weights $w\in \Re^E$ the shortest path metric of $(G,w)$ is denoted by $d_{(G,w)}$, 
and a linear ordering $v_1,\dots, v_n$ of $V$ is a {\em distance-preserving elimination ordering} of $(G,w)$ if for each $i\in [n]$ 
the weighted subgraph $(G_i, w)$ is {\em isometric}, i.e., 
$d_{(G_i,w)}$ coincides with the restriction of $d_{(G,w)}$ to $\{v_i,\cdots,v_n\}$.

We may identify the weighted graph $(G,w)$ with the symmetric matrix $W\in \mathcal S^V$ given by $W_{xy}=w_{xy}$ for $\{x,y\}\in E(G)$ and $W_{xy}=M$ for $\{x,y\} \notin E$ for some sufficiently large positive number $M$. 

\begin{proposition}
Let $(G,w)$ be a graph with nonnegative edge weights $w$ and corresponding weight matrix $W\in \mathcal S^V$.
Any perfect elimination ordering of  $-W$   is a distance-preserving elimination ordering of $(G,w)$.
\end{proposition}

\begin{proof}
Let $\pi=(v_1,\dots, v_n)$ be a perfect elimination ordering of $-W$ and assume $\pi$ is not distance-preserving for $(G,w)$.
Let $i$ be the smallest integer such that $d_{(G_{i+1}, w)}$ is not equal to the restriction of $d_{(G,w)}$ to  $\{v_{i+1}, \dots, v_n\}$.
Then there exist $v_j, v_k$ with $i<j<k$ such that $d_{(G_{i+1},w)}(v_j,v_k)>d_{(G_i,w)}(v_j,v_k)$ and thus every shortest path $P$ between $v_j$ and $v_k$ in $(G_i,w)$ passes through $v_i$. Let $P$ be such a path, say
$P=(v_j,\cdots,v_{r},v_i,v_{s},\cdots,v_k)$.
As $\pi$ is a perfect elimination ordering of $-W$ we have $W_{v_{r}v_{s}}\le \max\{W_{v_iv_{r}},W_{v_iv_{s}}\}$ and thus
$$d_{(G_i,w)}(v_{r},v_{s}) \le W_{v_{r}v_{s}}\le \max\{W_{v_iv_{r}},W_{v_iv_{s}}\}\le W_{v_iv_r}+ W_{v_iv_s} =
d_{(G_i,w)}(v_r,v_s).$$
Hence equality holds throughout and thus the path $P\setminus \{v_i\}$ is again a shortest path from $v_j$ to $v_k$  in $(G_i,w)$ but now not traversing $v_i$, a contradiction.
 \qed
\end{proof}
The reverse direction is not true in general, even for $\{0,1\}$ matrices.

\subsubsection{Outlook about other structured matrices and recognition algorithms}


We conclude with some observations about possible further extensions of graph properties to matrices and about recognition algorithms.

We present in this paper a matrix analogue of chordal graphs, motivated by the fact that Robinsonian matrices give a matrix analogue of  unit interval graphs. The key point in both cases is that chordal  and unit interval graphs can be defined by a three-points condition on their adjacency matrix. We now mention two more graph classes that would also fit within this pattern: interval  and cocomparability graphs. 

Recall that a graph $G=(V,E)$ is an {\em interval graph} if and only if there is a linear ordering $\pi$ of $V$ such that 
$\{x,z\}\in E$ implies  $\{y,z\}\in E$ for all  $ x<_{\pi} y <_{\pi} z$~\cite{Olariu91}.
Hence one may define an {\em interval matrix} $A$ to be a matrix whose index set admits a linear ordering $\pi$ such that 
\begin{equation}\label{eqI}
A_{xz}\leq A_{yz}  \quad \text{for all }  x<_{\pi} y <_{\pi} z.
\end{equation}
Similarly  $G$ is a {\em cocomparability graph} if and only if there is a linear ordering $\pi$ of $V$ such that 
$\{x,z\}\in E$ implies  $\{x,y\}\in E$ or $\{y,z\}\in E$ for all  $x<_{\pi} y <_{\pi} z$ \cite{Kratsch93}.
Hence one may define a {\em cocomparability matrix} $A$ to be a matrix whose index set admits a linear ordering $\pi$ such that 
\begin{equation}\label{eqCC}
A_{xz}\leq \max\{A_{xy}, A_{yz}\}  \quad \text{for all } x<_{\pi} y <_{\pi} z.
\end{equation}
Clearly relation (\ref{eq:rob}) implies (\ref{eqI}), which in turn implies both (\ref{eqineq}) and (\ref{eqCC}). In other words,
this extends to matrices the well known fact that unit interval graphs are interval graphs, which in turn are chordal and cocomparability graphs.

As shown in \cite{Laurent17} the structural characterization of  unit interval graphs in terms of minimal forbidden structures extends naturally  to the matrix setting and in this paper (Theorem \ref{thmA}) we provide such an extension for chordal graphs.
Establishing such extensions for interval  and cocomparability matrices,  
or a more general theory for generalizing structural characterizations from  graphs to matrices, 
is an interesting open problem which we leave for further research.


\medskip
There are two well-known linear time algorithms for recognizing chordal graphs (and finding perfect elimination orderings):
{\em lexicographic breadth-first search} (Lex-BFS)~\cite{Rose76} and {\em maximum cardinality search} (MCS) \cite{Tarjan84}. 
A natural question is whether these algorithmic techniques can be extended to matrices.

Corneil \cite{Corneil04} gives an algorithm for recognizing unit interval graphs based on three sweeps of Lex-BFS. 
In \cite{Laurent16} a weighted generalization of Lex-BFS, called {\em Similarity First Search (SFS)}, is introduced, which applies to symmetric matrices. It is shown in \cite{Laurent16}
 that $n$ sweeps of SFS can recognize Robinsonian matrices of size $n$ by returning a Robinson ordering.
It is natural to ask whether  SFS can also be used to find  perfect elimination orderings.

In \cite{Brandstadt97} it is shown that Lex-BFS can find a common perfect elimination ordering of the powers of a chordal graph $G$ (assuming $G^2$ is chordal).
Hence, in view of Corollary~\ref{cor:power}, if $D$ is a shortest path distance matrix, 
then  a single sweep of SFS finds a perfect elimination ordering of  $-D$.
However it is not difficult to construct a symmetric matrix for which a single sweep of SFS does not suffice to find  a perfect elimination ordering.
This thus raises the question whether 
one can  find  perfect elimination orderings of matrices using multiple sweeps of SFS.

Finally a generalization of MCS is proposed in the proof of Theorem~\ref{thm:power} in \cite{Brandstadt96},
which  can be adapted to the matrix setting. But it is not difficult to find a symmetric matrix for which this generalized MCS cannot find  a perfect elimination ordering in just one sweep.
Again  one may ask  whether a multi-sweep type variant of MCS  can find a perfect elimination ordering.

Note also that it is shown recently in \cite{Dusart} that $n$ sweeps of Lex-BFS permit to find elimination orderings certifying cocomparability graphs.

\section{Proof of Theorem~\ref{thmA}}\label{secproof}
In this section we prove Theorem~\ref{thmA}.
By Lemma~\ref{prop:1}, if a matrix $A$ contains a self-contained pair of weighted chordless walks then it has no perfect elimination ordering, hence
it remains to show the converse implication.
A first easy observation is that it in fact suffices to show the existence of a simplicial vertex. Indeed,  Theorem \ref{thmA} follows easily from the next result (using induction on the size of the matrix).

\begin{theorem}\label{thmB}
If a matrix $A$ has no self-contained pair of weighted chordless walks then $A$ has a simplicial vertex.
\end{theorem}

We will in fact prove a stronger result (Theorem \ref{thmC} below). 
Before stating this stronger result we introduce some notation and preliminary facts.
Throughout we let $A$ be a symmetric matrix indexed by a finite set $V$.

\begin{definition}\label{defseparation}
Set $\min A=\min\{A_{xy}: x\ne y \in V\}$. Given $X,Y\subseteq V$ we say that $(X,Y)$ is a {\em separation} of $A$ if $X\setminus Y, Y\setminus X\ne \emptyset$ and $A_{xy}=\min A$ for all $x\in X\setminus Y$ and $y\in Y\setminus X$.
\end{definition}

\begin{lemma}\label{lemsep}
Let $(X,Y)$ be a separation of $A$. If $x\in X\setminus Y$ is a simplicial vertex of $A[X]$ then $x$ is a simplicial vertex of $A$,
where $A[X]$ denotes the principal submatrix of $A$ indexed by $X$.
\end{lemma}

\begin{proof}
Let $u,v\in V$, we show that $A_{uv} \ge \min\{A_{ux},A_{vx}\}$. This is true when (say) $u\in Y\setminus X$ since then $A_{xu}=\min A$, and also when $u,v\in X$ because $x$ is simplicial in $A[X]$.\qed
\end{proof}

\begin{lemma}\label{lempath}
There exists a separation $(X,Y)$ of $A$ for which the following property holds for each $Z\in \{X,Y\}$:
\begin{equation}\tag{P}
\begin{array}{l}
\text{For all } u\in Z\setminus (X\cap Y) \text{ and } s\in X\cap Y, \text{ either  } A_{su}>  \min A  \text{ holds}, \\
\text{or there exists a weighted chordless walk from } u \text{ to } s \text{ in } A[Z] \\
\text{which is  internally vertex-disjoint from } X\cap Y.
\end{array}
\end{equation}
In addition, given $a,b\in V$ with $A_{ab}=\min A$, there exists a separation $(X,Y)$ of $A$ separating $a,b$ (i.e., 
$a\in X\setminus Y$, $b\in Y\setminus X$ or vice versa) and (P) holds for $u\in \{a,b\}$.
\end{lemma}

\begin{proof}
Let $G=(V,E)$ be the graph on $V$ whose edges are the pairs $\{x,y\}$ with $A_{xy}>\min A$.
If $G$ is not connected and $C_1,\cdots, C_t$ ($t\ge 2$) denote its  connected components then we may set 
$X=C_1$ and $Y=V\setminus C_1$. 

Assume now that $G$ is connected. Let $S$ be a minimal vertex separator of $G$ and let $C_1,\cdots,C_t$ be the connected components of $G[V\setminus S]$. Fix $s\in S$. As $S\setminus \{s\}$ is a not a vertex separator of $G$  it follows that $s$ is adjacent to at least one vertex in each component $C_i$. Hence for any $x\in C_i$ there is a path from $x$ to $s$ in $G[C_i\cup\{s\}]$ and if we choose this path shortest possible then either it consists of a single edge or it provides a weighted chordless walk from $x$ to $s$ in $A$ which is contained in $C_i\cup\{s\}$ and thus internally vertex-disjoint from $S$. Thus the lemma holds if we set, e.g., $X=C_1$ and $Y=V\setminus C_1$.

Finally  if we are given a pair $a,b$ with $A_{ab}=\min A$ then choosing $S$ to be a minimal $(a,b)-$vertex separator in $G$ and   $C_1$   the component containing $a$ gives the final statement. \qed
\end{proof}

\begin{definition}\label{defcritical}
A walk $W$ is said to be a {\em critical walk} of $A$ if 
$W$ is a closed weighted chordless walk whose end point $v_0$ is simplicial in $A$ and 
there exists an internal element $u\in I(W)$ such that $A_{v_0u}=\min A$.

A walk $W$ is said to be {\em rooted in a set $S\subset V$} if  the end points of $W$ belong to $S$ and the internal elements of $W$ belong to $V\setminus S$ with $I(W)\neq \emptyset$.
\end{definition}

We can now formulate the following stronger result.

\begin{theorem}\label{thmC}
If a matrix $A$ has no self-contained pair of weighted chordless walks then it satisfies at least one of the following two properties (A) or (B):
\begin{equation}\tag{A}
A \text{  has a critical walk,}
\end{equation}
\begin{equation}\tag{B}
A \text{ has two distinct simplicial vertices } u,v \text{ such that } A_{uv}=\min A.
\end{equation}
\end{theorem}

Property (B) is a matrix analogue of a known fact that the diameter of a chordal graph is attained by a pair of simplicial vertices (see, e.g., ~\cite{Farber86}).
This property is no longer true for symmetric matrices, see Figure~\ref{fig:3}.
A weaker well known fact by Dirac~\cite{Dirac61} is  that a chordal graph has at least two simplicial vertices that are not adjacent if it is not a complete graph. 
The example in Figure~\ref{fig:3} shows that even this weaker fact fails for general matrices.

\begin{figure}[t]
\centering
\begin{tikzpicture}{scale=0.85}

\def \n {4}
\def \radius {1cm}
\foreach \s in {1,...,\n}
{
  \node[draw, circle, inner sep=1pt,minimum size=1.5pt] (\s) at ({360/\n * (\s - 1)+90}:\radius) {$\s$};
}
\foreach  \x/\y in {1/2,1/3}
{
\draw[thick] (\x) -- (\y);
}
\foreach \x/\y in {2/3,2/4,3/4}
{
\draw[thick,dashed] (\x) -- (\y);
}
\end{tikzpicture}
\caption{A visualization of a $\{0,1,2\}$-matrix of size four as a weighted graph, 
where each bold edge has weight two and each dashed edge has weight one.
Observe that $4$ is the unique simplicial vertex in $A$.}
\label{fig:3}
\end{figure}
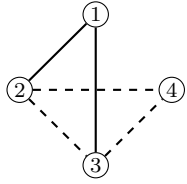

Clearly both properties (A) and (B) imply the existence of a simplicial vertex, hence Theorem \ref{thmC} does indeed imply Theorem \ref{thmB} (and thus in turn Theorem \ref{thmA}). The following lemma will provide  the main technical ingredient for the proof of Theorem~\ref{thmC}.

\begin{lemma}\label{lemP12}
Let $(X,Y)$ be a separation of $A$ satisfying the property (P) from Lemma~\ref{lempath}. 
Assume 
that  every proper (i.e., distinct from $A$) principal submatrix of $A$ satisfies (A) or (B). 
Then each  $Z\in \{X,Y\}$ satisfies at least one of the following two properties (P1), (P2):
\begin{equation}\tag{P1}
A[Z] \text{ has a simplicial vertex belonging to } Z\setminus (X\cap Y),
\end{equation}
\begin{equation}\tag{P2}
\begin{array}{l}
\text{there is a weighted chordless walk in } A[Z] \text{ which is rooted in } X\cap Y. 
\end{array}
\end{equation}
\end{lemma}



\begin{proof} 
We show that $A[X]$ satisfies (P1) or (P2) (same reasoning for $A[Y]$).
For this we will iteratively construct a sequence of subsets $Z_0=X,Z_1,\cdots,Z_k$ which is strictly monotone:
$Z_0 \supset \cdots \supset Z_i\supset \cdots \supset Z_k$ and satisfies  the following two properties (Q1)-(Q2)
for all $0\le i\le k$:
\begin{equation}\tag{Q1}
Z_i \text{  meets } Y \text{ and } V\setminus Y,
\end{equation}
\begin{equation}\tag{Q2}
\text{if } x\in Z_i\setminus Y \text{ is simplicial in }
A[Z_i] \text{ then } x \text{ is also simplicial in } A[X].
\end{equation}

We first observe that if we can find a set $Z_k$ satisfying (Q1)-(Q2) and  $|Z_k\cap Y|=1$  then we can stop and conclude that (P1) or (P2) holds for $A[X]$. To see this  consider the (proper) submatrix $A[Z_k]$, which by assumption 
 satisfies (A) or (B). Assume first  $A[Z_k]$ satisfies (B). Then there are distinct simplicial elements $u,v$ in $A[Z_k]$. At least one of them, say $u$, belongs to $Z_k\setminus Y$. Then by (Q2) we know that $u$ is simplicial in $A[X]$ and thus (P1) holds for $A[X]$. Assume now $A[Z_k]$ satisfies (A). Then there is a critical walk $W$ in $A[Z_k]$.
If its end point  $v_0$ belongs to $Z_k\setminus Y$ then $v_0$ is critical in $A[X]$ (again by (Q2)) and thus (P1) holds for $A[X]$. Assume now $v_0\in Y$. Then as $|Z_k\cap Y|=1$ the walk $W$ is in fact a weighted chordless walk rooted in $X\cap Y$ and thus (P2) holds for $A[X]$.

\smallskip 
We now proceed to  construct the sets $Z_i$ satisfying (Q1)-(Q2) until we can conclude that (P1) or (P2) holds for $A[X]$. 
We start with $Z_0=X$ which indeed satisfies (Q1)-(Q2).
Suppose we have $Z_{i-1}$ satisfying (Q1)-(Q2) and  $|Z_{i-1}\cap Y|\ge 2$.
Consider the matrix $A[Z_{i-1}]$.
%
%
We claim that if (P1) and  (P2) do not hold for $A[X]$ then 
\begin{equation}\label{eqCD}
\min A[Z_{i-1}]=A_{vw} \text{ for  some  }
v, w \in Z_{i-1}\cap Y \text{ with } v\neq w.
\end{equation}
By  assumption  $A[Z_{i-1}]$ satisfies (A) or (B). 
Assume first (B) holds and let $v, w$ be simplicial vertices in $A[Z_{i-1}]$ with $A_{vw}=\min A[Z_{i-1}]$.
If at least one of the two vertices belongs to $Z_{i-1}\setminus Y$, then (P1) follows by (Q2).
Otherwise, as $A_{uv}=\min A[Z_{i-1}]$, we get (\ref{eqCD}).

Assume next $A[Z_{i-1}]$ satisfies (A) and let $W$ be a critical walk in $A[Z_{i-1}]$, so its  end point $v_0$ is simplicial in $A[Z_{i-1}]$. If $v_0\in Z_{i-1}\setminus Y$ then $v_0$ is simplicial in $A[X]$ (by (Q2)) and thus (P1) holds. Assume now $v_0\in Z_{i-1}\cap Y$. 
As $W$ is critical there exists an internal vertex $u\in I(W)$ such that $A_{v_0u}=\min A[Z_{i-1}]$. If $u\in Y$ then (\ref{eqCD}) holds. Otherwise $W$ contains a subwalk which is a weighted chordless walk rooted in $X\cap Y$and thus  (P2) holds.

So we may now assume  (\ref{eqCD}) holds.
Let $(C,D)$ be a separation of $A[Z_{i-1}]$  separating  $v$ and $w$, as in Lemma \ref{lempath}, with (say) $v\in C\setminus D$ and  $w\in D\setminus C$. Without loss of generality $C\cap (V\setminus Y)\ne \emptyset$. Set $Z_i=C$. Then $Z_i\subset Z_{i-1}$ (since $w\in Z_{i-1}\setminus C$) and (Q1) holds for $Z_i$.
We claim:
\begin{equation}\label{eqCD2}
\text{If } (Z_{i-1}\setminus Y) \cap (C\cap D) \ne \emptyset\  \text{ then (P2) holds for } A[X].
\end{equation}
For this consider $z\in (Z_{i-1}\setminus Y) \cap (C\cap D)$
and  apply Lemma \ref{lempath}  to the separation $(C,D)$ of $A[Z_{i-1}]$ separating  $v,w\in Z_{i-1}\setminus (C\cap D)$ and $z\in C\cap D$.
Then either 
$A_{zv}> \min A[Z_{i-1}]$ holds or there exists a weighted chordless walk $W_1=(z,a,\cdots,v)$ from $z$ to $v$  in $A[Z_{i-1}]$ that is internally vertex-disjoint from $C\cap D$; in the former case we set $W_1=(z,v)$ (i.e., $a=v$). 
Analogously, either  $A_{zw}> \min A[Z_{i-1}]$ holds or there exists a weighted chordless walk $W_2=(z,b,\cdots, w)$ from $z$ to $w$ internally vertex-disjoint from $C\cap D$; in the former case set $W_2=(z,w)$ (ie., $b=w$).
Then $a\in C\setminus D$ and $b\in D\setminus C$, which  implies 
$A_{ab}=\min A[Z_{i-1}]$. 
From this it follows that the walk $W$ obtained by traveling first from $v$ to $z$ along the reverse of $W_1$ and then from $z$ to $w$ along $W_2$ is a weighted chordless walk in $A[X]$.  
By $z\notin X\cap Y$ and $v, w\in X\cap Y$ it contains at least one subwalk $W_0$ which is a weighted chordless walk in $A[X]$ rooted in $X\cap Y$ and thus   (P2) holds.

So we may now assume in addition that $(Z_{i-1}\setminus Y) \cap (C\cap D)=\emptyset$, we claim that $Z_i=C$ satisfies (Q2). For this let $x\in C\setminus Y$ simplicial in $A[C]$, we show that $x$ is simplicial in $A[X]$. Indeed,  $x\not\in D$ and thus $x$ is simplicial in $A[Z_{i-1}]$ (by Lemma \ref{lemsep}) and also in $A[X]$ (as $A[Z_{i-1}]$ satisfies  (Q2)). Hence $Z_i=C$ satisfies (Q1)-(Q2), which concludes the proof. \qed
\end{proof}

With the help of Lemma \ref{lemP12} we can now prove Theorem \ref{thmC}.

\begin{proof}{\em (of Theorem \ref{thmC})}
The proof is by induction on the size of the matrix $A$.
So we may assume $A$ has no self-contained pair of weighted chordless walks and every proper principal submatrix of $A$ satisfies (A) or (B).
Let $(X,Y)$ be a separation of $A$ satisfying property (P) of Lemma \ref{lempath}.

Assume first  $X\cap Y=\emptyset$. By the induction assumption $A[X]$ satisfies (A) or (B), which implies that the same holds for $A$ (since a simplicial vertex of $A[X]$ is also  simplicial in $A$ in view of Lemma \ref{lemsep}).

Assume now $S=X\cap Y\ne \emptyset$.
In view of Lemma \ref{lemP12}  both $A[X]$ and $A[Y]$ satisfy (P1) or (P2). We distinguish three cases, depending on these possible combinations.

\medskip\noindent
{\bf Case 1: Both $A[X]$ and $A[Y]$ satisfy (P1).} Hence there exist $x\in X\setminus S$ which is simplicial in $A[X]$ and $y\in Y\setminus S$ which is simplicial in $A[Y]$. Then $x,y$ are simplicial in $A$  (by Lemma \ref{lemsep})  with $A_{xy}=\min A$ and thus  (B) holds.

\medskip\noindent
{\bf Case 2: $A[X]$ satisfies (P1) and $A[Y]$ satisfies (P2) (or vice versa).}
So let $x\in X\setminus S$ which is simplicial in $A[X]$ and let $Q=(v_1,v_2,\cdots,v_{k-1},v_k)$ be a chordless walk in $A[Y]$ which is rooted in $S$ (i.e., $v_1,v_k\in S$ and $v_2,\cdots,v_{k-1}\in Y\setminus S$ with $k\geq 3$).
By property (P) there exist  weighted chordless walks $W_1=(x,\cdots,u,v_1)$ from $x$ to $v_1$  (resp., $W_2=(x,\cdots, v,v_k)$ from $x$ to $v_k$) 
in $A[X]$ which are internally vertex-disjoint from $S$,
where we allow a walk $W_1=(x,v_1)$ (resp., $W_2=(x,v_k)$) of length one, in which case $A_{xv_1}>\min A$ 
(resp., $A_{xv_k}>\min A$).
Consider the walk $W$ obtained by concatenating the three walks $W_1,Q,W_2$ in that order, that we may visualize as $W=([x] W_1 [v_1] Q [v_k] W_2 [x])$ (where we  insert the connection vertices between consecutive walks into brackets just to clarify the definition). Then $W$ is a closed walk whose end point $x$ is indeed simplicial in $A$ (in view of Lemma \ref{lemsep}). Moreover $W$ is a weighted chordless walk in $A$. Indeed the only missing inequalities are 
$A_{uv_2}<\min\{A_{uv_1},A_{v_1v_2}\}$ and $A_{vv_{k-1}}<\min \{A_{vv_k},A_{v_kv_{k-1}}\}$ which do hold since 
$A_{uv_2}=A_{vv_{k-1}}=\min A$ (as $u,v\in X\setminus S$ and $v_2,v_{k-1}\in Y\setminus S$).
Finally, we have $A_{xv_2}=\min A$. Therefore $W$ is a critical walk in $A$ and thus (A) holds.

\medskip\noindent
{\bf Case 3: Both $A[X]$ and $A[Y]$ satisfy (P2).}
So let $P=(u_1,u_2,\cdots, u_{l-1},u_l)$  (resp., $Q=(v_1,v_2,\cdots,v_{k-1},v_k)$) be a chordless walk in $A[X]$ (resp., in $A[Y]$),  which are rooted in $S$ (i.e., $u_1,u_l,v_1,v_k\in S$, $u_2,\cdots,u_{l-1}\in X\setminus S$, and $v_2,\cdots,v_{k-1}\in Y\setminus S$ with $k, l\geq 3$).
By property (P) there exist weighted chordless walks $W_1=(v_2,\cdots, y,u_1)$ from $v_2$ to $u_1$  and 
$W_2= (v_2,\cdots,y',u_l)$ from $v_2$ to $u_l$ in $A[Y]$ which are internally disjoint from $S$ (where $W_1$ and $W_2$ may have length one as in Case 2). Then one can check (as in Case 2) that the concatenated walk $W=([v_2] W_1 [u_1] P [u_l] W_2 [v_2])$ is a closed weighted chordless walk with $v_2$ as only vertex which is not an internal element of $W$.
Analogously, using again (P) we find weighted chordless walks $W_3=(u_2,\cdots,x,v_1)$ and
$W_4=(u_2,\cdots,x',v_k)$ in $A[X]$ which are internally disjoint from $S$.
So the walk $W'=([u_2] W_3 [v_1] Q [v_k] W_4 [u_2])$ is a weighted chordless walk in $A$ with only $u_2$ as non-internal element. Finally as $v_2$ is an internal element of $W'$ and $u_2$ is an internal element of $W$, the two walks $(W,W')$ form a self-contained pair of weighted chordless walks in $A$, which contradicts the assumption on $A$.
So we reach a contradiction and the proof is completed.\qed \end{proof}

\begin{acknowledgements}
 The second author was supported by JSPS Postdoctoral Fellowships for Research Abroad.
\end{acknowledgements}



\end{document}